\documentclass[11pt]{article}
\usepackage{authblk}
\usepackage{amssymb}
\usepackage[all]{xy}
\usepackage{mathrsfs}
\usepackage{amsmath,amssymb,amsthm}
\usepackage{extarrows}
\usepackage[mathscr]{eucal}
\usepackage{mathrsfs}
\newtheorem{theorem}{Theorem}[section]
\newtheorem{lemma}[theorem]{Lemma}

\newtheorem{question}[theorem]{Question}
\newtheorem{remark}[theorem]{Remark}
\newtheorem{corollary}[theorem]{Corollary}
\newtheorem{example}[theorem]{Example}
\newtheorem{problem}[theorem]{Problem}

\def\to{\rightarrow}
\def\f{\mathfrak}

\def\c{\mathcal}
\def\b{\mathbf}
\def\r{\mathrm}
\def\bb{\mathbb}

\date{}
\begin{document}
\baselineskip16pt
\title{Some Inequalities for Continuous Algebra-Multiplications on a Banach Space}
\author{Maysam Maysami Sadr\thanks{sadr@iasbs.ac.ir}}
\affil{Department of Mathematics, Institute for Advanced Studies in Basic Sciences, Zanjan, Iran}
\maketitle
\begin{abstract}
In this short note, we first consider some inequalities
for comparison of some algebraic properties of two continuous algebra-multiplications on an arbitrary Banach space and then,
as an application, we consider some very basic observations on the space of all continuous algebra-multiplications on a Banach space.



\end{abstract}
\section{Introduction}
One of the fundamental and hardest questions in the theory of Banach algebras is the following.
\begin{question}
Let $A$ be a Banach space and let $\star$ and $\diamond$ be two Banach algebra multiplications on $A$. How are algebraic
properties of the Banach algebras $(A,\star)$ and $(A,\diamond)$ related if $\star$ and $\diamond$ be sufficiently close?
\end{question}
This question which is the central subject of perturbation theory of Banach algebras, has been considered by many authors.
See, for instance, \cite{Jarosz1,Johnson1} and references therein.

In this short note, we first consider some inequalities
for comparison of some algebraic properties of two continuous algebra-multiplications on an arbitrary Banach space and then,
as an application, we consider some very basic observations on the space of all continuous algebra-multiplications on a Banach space.

\textbf{Conventions \& Notations.}
Operator norm of bounded linear and bilinear operators, all are denoted by $|\cdot|$.
Throughout,  $E$ denotes a non-zero Banach space. The Banach space of bounded linear operators on $E$ is denoted by $\r{B}(E)$, and
the Banach space of bounded bilinear operators from $E\times E$ into $E$ is denoted by $\r{B}_2(E)$.
Suppose that $*\in\r{B}_2(E)$ is associative as a binary operation on $E$.
Then $*$ is called continuous multiplication on $E$.
We denote by $\r{Mltp}(E)\subset\r{B}_2(E)$ the set of all continuous multiplications on $E$.
For $*\in\r{Mltp}(E)$, the pair $(E,*)$ is called complete normed algebra.
If $|*|\leq1$, then $*$ is called Banach algebra multiplication and $(E,*)$
is called Banach algebra. A multiplication $*\in\r{Mltp}(E)$ is called unital (resp. commutative) if the algebra $(E,*)$
is unital (resp. commutative).  We denote by $\r{Mltp}_\r{u}(E)\subset\r{Mltp}(E)$ (resp. $\r{Mltp}_\r{c}(E)\subset\r{Mltp}(E)$)
the set of unital (resp. commutative) continuous multiplications on $E$. For $*\in\r{Mltp}_\r{u}(E)$
we denote by $e_*$ the unite element of $(E,*)$.

For a complete normed algebra $(A,*)$ the products of elements of $A$, if there is no danger of misunderstanding about
the multiplication $*$, is denoted by juxtaposition. In the case that $(A,*)$ is unital, the inverse of an invertible element $a\in A$ is denoted
by $a^{-1}$ or by $a^{-1}_*$ in order to indicate that the inverse is with respect to $*$.
Also, the set of invertible elements of $(A,*)$ is denoted by $\r{Inv}(A,*)$, and for any $a\in A$, $\sigma_*(a)$ and $\rho_*(a)$
denote respectively the spectrum and resolvent of $a$.
\section{Some Inequalities}\label{1804022216}
We begin by a simple generalization of a famous lemma in Banach Algebra Theory
(\cite[Lemma 1.2.5]{BonsallDuncan1},\cite[Theorem 2.1.29]{Dales1},\cite[Theorem 10.11]{Rudin1}).
\begin{theorem}\label{1803162128}
Let $(A,*)$ be a unital complete normed algebra with unit $e=e_*$. Let $a\in A$ be invertible, $r$ be a numerical constant with
$0\leq r<1$, and let $b\in A$ be such that
$$\|b-a\|\leq r(\|a^{-1}\||*|^2)^{-1}.$$
Then, $b$ is invertible. Moreover,
$$\|b^{-1}-a^{-1}\|\leq(\|e\||*|+r(1-r)^{-1})|*|^2\|a^{-1}\|^2\|b-a\|.$$
\end{theorem}
\begin{proof}
Let $c:=e-a^{-1}b=a^{-1}(a-b)$. We have $\|c\|\leq\|b-a\|\|a^{-1}\||*|\leq r|*|^{-1}$. Thus,
$$|*|\sum_{n=1}^\infty\|c^n\|\leq\sum_{n=1}^\infty(|*|\|c\|)^n\leq\sum_{n=1}^\infty r^{n}=r(1-r)^{-1}.$$
This shows that $(\sum_{n=0}^\infty c^n)\in A$ is the inverse of $a^{-1}b=e-c$. It follows that $b\in\r{Inv}(A,*)$. Moreover,
$b^{-1}=\sum_{n=0}^\infty c^na^{-1}$, and we have,
\begin{equation*}
\begin{split}
\|b^{-1}-a^{-1}\|&=\|\sum_{n=1}^\infty c^na^{-1}\|\\
&\leq|*|^2\|a^{-1}\|\|c\|\|\sum_{n=0}^\infty c^n\|\\
&\leq|*|^2\|a^{-1}\|\|c\|(\|e\|+\sum_{n=1}^\infty\|c^n\|)\\
&\leq(\|e\||*|+r(1-r)^{-1})|*|^2\|a^{-1}\|^2\|b-a\|.
\end{split}
\end{equation*}
\end{proof}
We introduce some notations: Let $(E,*)$ be a complete normed algebra.
For any $x\in E$, let $\f{l}_{*,x}\in\r{B}(E)$ denote the left multiplication operator by $x$, i.e. $\f{l}_{*,x}(y)=x*y$. Similarly,
we let $\f{r}_{*,x}\in\r{B}(E)$ denote the right multiplication operator by $x$.
If $(E,*)$ is unital then we have,
$$\|x\|\|e_*\|^{-1}\leq|\f{l}_{*,x}|,|\f{r}_{*,x}|\leq|*|\|x\|.$$
Also, it is easily checked that $x\in\r{Inv}(E,*)$ iff both of $\f{l}_{*,x}$ and $\f{r}_{*,x}$ are invertible in the Banach algebra $\r{B}(E)$.
\begin{theorem}\label{1803312216}
Let $*,\diamond\in\r{Mltp}(E)$. Suppose that $*$ is unital and
for a numerical constant $r$ with $0\leq r<1$ we have $|\diamond-*|\leq r\|e_*\|^{-1}$.
Then, $\diamond$ is also unital. Moreover, letting $s=(1+r(1-r)^{-1})$, we have
\begin{equation}\label{1804052103}
\|e_\diamond-e_*\|\leq s|\diamond-*|\|e_*\|^2\leq rs\|e_*\|.
\end{equation}
\end{theorem}
\begin{proof}
We have
\begin{equation}\label{1804022112}
|\f{l}_{\diamond,e_*}-\r{id}|=|\f{l}_{\diamond,e_*}-\f{l}_{*,e_*}|\leq|\diamond-*|\|e_*\|\leq r.
\end{equation}
Thus, by Theorem \ref{1803162128}, $\f{l}_{\diamond,e_*}$ is invertible in $\r{B}(E)$.
Similarly, it is proved that $\f{r}_{\diamond,e_*}$ is invertible in $\r{B}(E)$.
By invertibility of $\f{l}_{\diamond,e_*}$, there is $e'\in E$ such that $e_*\diamond e'=e_*$. By invertibility of
$\f{r}_{\diamond,e_*}$, for every $x\in E$ there is $x'$ such that $x=x'\diamond e_*$. Thus, we have
$$x\diamond e'=(x'\diamond e_*)\diamond e'=x'\diamond(e_*\diamond e')=x'\diamond e_*=x.$$
This shows that $e'$ is a right unit of $(E,\diamond)$. Similarly, it is proved that $(E,\diamond)$ has a left unit. Thus, $(E,\diamond)$
is unital and $e_\diamond=e'$. It follows from (\ref{1804022112}) and Theorem \ref{1803162128} that
$$|\f{l}_{\diamond,e_*}^{-1}-\r{id}|\leq s|\diamond-*|\|e_*\|\leq rs.$$
Thus, $\|e_\diamond-e_*\|=\|\f{l}_{\diamond,e_*}^{-1}(e_*)-\r{id}(e_*)\|\leq s|\diamond-*|\|e_*\|^2\leq rs\|e_*\|$.
\end{proof}
We have the following easy estimate of distance between unit elements.
\begin{theorem}\label{1804052039}
Let $*,\diamond\in\r{Mltp}_\r{u}(E)$. Then,
$$\|e_\diamond-e_*\|\leq|\diamond-*|\|e_\diamond\|\|e_*\|.$$
\end{theorem}
\begin{proof}
We have, $\|e_\diamond-e_*\|=\|e_\diamond*e_*-e_\diamond\diamond e_*\|\leq|\diamond-*|\|e_\diamond\|\|e_*\|$.
\end{proof}
\begin{theorem}\label{1803162143}
Let $*\in\r{Mltp}_\r{u}(E)$, and let $a\in\r{Inv}(E,*)$.
Suppose that $\diamond\in\r{Mltp}(E)$ is such that $|\diamond-*|<(\|a\|\|a^{-1}_*\||*|)^{-1}$.
Then $\diamond$ is unital and $a\in\r{Inv}(E,\diamond)$. Moreover, for every $M>0$ there exists a universal constant
$C_M>0$ such that if
\begin{equation}\label{1804060138}
|*|,\|e_*\|,\|a\|,\|a_*^{-1}\|\leq M\hspace{2mm}\text{and}\hspace{2mm}|\diamond-*|\leq rM^{-3}
\end{equation}
for some constant $r$ with $0\leq r<1$, then $\diamond$ is unital, $a\in\r{Inv}(E,\diamond)$, and
$$\|a^{-1}_\diamond-a^{-1}_*\|\leq s^2C_M|\diamond-*|,$$
where $s=1+r(1-r)^{-1}$.
\end{theorem}
\begin{proof}
Suppose that $|\diamond-*|<(\|a\|\|a^{-1}_*\||*|)^{-1}$. Since $\|e_*\|\leq\|a\|\|a^{-1}_*\||*|$, we have
$|\diamond-*|<\|e_*\|^{-1}$. Thus, by Theorem  \ref{1803312216}, $\diamond$ is unital. We have,
\begin{equation}\label{1803162133}
|\f{l}_{\diamond,a}-\f{l}_{*,a}|\leq\|a\||\diamond-*|<(\|a^{-1}_*\||*|)^{-1}\leq |\f{l}^{-1}_{*,a}|^{-1}.
\end{equation}
Note that $\f{l}_{*,a_*^{-1}}=\f{l}^{-1}_{*,a}$.
From (\ref{1803162133}) and Theorem \ref{1803162128} it is concluded that $\f{l}_{\diamond,a}$ is invertible in $\r{B}(E)$.
Similarly, it is proved that $\f{r}_{\diamond,a}$ is invertible in $\r{B}(E)$. Thus, $a\in\r{Inv}(E,\diamond)$.

Suppose that $M>0$ be such that (\ref{1804060138}) is satisfied. We have $|\diamond-*|\leq rM^{-3}<(\|a\|\|a^{-1}_*\||*|)^{-1}$.
Thus, by the first part, $\diamond$ is unital and $a\in\r{Inv}(E,\diamond)$. Analogous to (\ref{1803162133}), we have
$$|\f{l}_{\diamond,a}-\f{l}_{*,a}|\leq r|\f{l}^{-1}_{*,a}|^{-1}.$$
Thus, by Theorem \ref{1803162128}, we have $|\f{l}^{-1}_{\diamond,a}-\f{l}^{-1}_{*,a}|\leq s|\f{l}^{-1}_{*,a}|^2|\f{l}_{\diamond,a}-\f{l}_{*,a}|$.
It follows that
\begin{equation}\label{1803250150}
|\f{l}^{-1}_{\diamond,a}-\f{l}^{-1}_{*,a}|\leq s|*|^2\|a_*^{-1}\|^2\|a\||\diamond-*|\leq sM^5|\diamond-*|
\end{equation}
We have,
\begin{equation}\label{1803250151}
\begin{split}
\|a^{-1}_\diamond-a^{-1}_*\|&=\|\f{l}^{-1}_{\diamond,a}(e_\diamond)-\f{l}^{-1}_{*,a}(e_*)\|\\
&\leq\|\f{l}^{-1}_{\diamond,a}(e_\diamond)-\f{l}^{-1}_{*,a}(e_\diamond)\|+\|\f{l}^{-1}_{*,a}(e_\diamond)-\f{l}^{-1}_{*,a}(e_*)\|\\
&\leq|\f{l}^{-1}_{\diamond,a}-\f{l}^{-1}_{*,a}|\|e_\diamond\|+|\f{l}^{-1}_{*,a}|\|e_\diamond-e_*\|.
\end{split}
\end{equation}
By Theorem \ref{1803312216}, we have
\begin{equation}\label{1804060243}
\|e_\diamond-e_*\|\leq s\|e_*\|^2|\diamond-*|\leq sM^2|\diamond-*|
\end{equation}
Also, we have,
\begin{equation}\label{1804060242}
\|e_\diamond\|\leq\|e_\diamond-e_*\|+\|e_*\|\hspace{2mm}\text{and}\hspace{2mm}|\f{l}^{-1}_{*,a}|\leq|*|\|a_*^{-1}\|\leq M^2.
\end{equation}
Now, the second part of the theorem follows from (\ref{1803250150}),(\ref{1803250151}),(\ref{1804060243}), and (\ref{1804060242}).
\end{proof}
The following is a direct corollary of Theorem  \ref{1803162143}.
\begin{corollary}
Let $(\star_n)_n$ be a sequence in $\r{Mltp}_\r{u}(E)$ converging to $\star\in\r{Mltp}_\r{u}(E)$. Then, for every $a\in E$ we have
$$\rho_\star(a)\subseteq\liminf\rho_{\star_n}(a):=\cup_{n}\cap_{k\geq n}\rho_{\star_k}(a),$$
or equivalently,
$$\sigma_\star(a)\supseteq\limsup\sigma_{\star_n}(a):=\cap_{n}\cup_{k\geq n}\sigma_{\star_k}(a).$$
\end{corollary}
\section{Elementary aspects of the geometry of $\r{Mltp}(E)$}\label{1804022217}
In this section, we consider some elementary observations on the geometry of $\r{Mltp}(E)$ as a
subset of the normed space $\r{B}_2(E)$.
\begin{lemma}\label{1901231707}
Let $(\star_n)_n$ be a sequence in $\r{B}_2(E)$ such that it converges to $\star\in\r{B}_2(E)$.
Let $(a_n)_n,(b_n)_n$ be two convergent sequence in $E$ such that $a_n\to a$ and $b_n\to b$. Then, $a_n\star_n b_n\to a\star b$.
\end{lemma}
\begin{proof}
It follows from the following simple inequality:
$$\|a_n\star_n b_n- a\star b\|\leq|\star_n-\star|\|a_n\|\|b_n\|+|\star|\|a_n\|\|b_n-b\|+|\star|\|b\|\|a_n-a\|.$$
\end{proof}
\begin{theorem}\label{1803191105}
$\r{Mltp}(E)$ is a closed subset of $\r{B}_2(E)$.
\end{theorem}
\begin{proof}
Let $(\star_n)_n$ be a sequence in $\r{Mltp}(E)$ such that it converges to $\star\in\r{B}_2(E)$.
We must show that $\star\in\r{Mltp}(E)$. Let $a,b,c\in E$. By permanent application of Lemma \ref{1901231707}, we have
$$(a\star b)\star c=(\lim_n a\star_nb)\star c=\lim_n (a\star_n b)\star_nc=\lim_n a\star_n (b\star_nc)=a\star(\lim_n b\star_n c)=a\star(b\star c).$$
Thus, $\star$ is associative and the proof is complete.
\end{proof}
\begin{theorem}\label{1803191105}
$\r{Mltp}_\r{u}(E)$ and $\r{Mltp}_\r{c}(E)$ are, respectively, open and closed in $\r{Mltp}(E)$.
\end{theorem}
\begin{proof}
It follows from Theorem \ref{1803312216} that $\r{Mltp}_\r{u}(E)$ is open in $\r{Mltp}(E)$. By Lemma \ref{1901231707},
it is easily verified that any limit point of $\r{Mltp}_\r{c}(E)$ is a commutative multiplication.
\end{proof}
\begin{theorem}\label{1901272153}
The map $\star\mapsto e_\star$ from $\r{Mltp}_\r{u}(E)$ to $E$ is continuous. Moreover, if $\star$ converges to a point in the boundary of
$\r{Mltp}_\r{u}(E)$ in $\r{Mltp}(E)$, then $\|e_\star\|$ converges to $\infty$.
\end{theorem}
\begin{proof}
The first part follows directly from Inequality (\ref{1804052103}). For the second part, let $(\star_n)_n$ be a sequence in
$\r{Mltp}_\r{u}(E)$ converging to $\star$ such that $\star$ is in the boundary of $\r{Mltp}_\r{u}(E)$ in $\r{Mltp}(E)$. Since
$\r{Mltp}_\r{u}(E)$ is open in $\r{Mltp}(E)$, $\star$ is not unital. If the sequence $\|e_{\star_n}\|$ is bounded, then Theorem \ref{1803312216}
shows that $\star$ must be unital. The proof is complete.
\end{proof}
\begin{theorem}
Let $a$ be a nonzero element of $E$. Let $U_a:=\{\star\in\r{Mltp}(E):a\in\r{Inv}(E,\star)\}$.
Then $U_a$ is open in $\r{Mltp}_\r{u}(E)$ and the map $\star\mapsto a_\star^{-1}$ from $U_a$ to $E$ is continuous.
\end{theorem}
\begin{proof}
It follows easily from Theorem \ref{1803162143}.
\end{proof}
Let us mention that for finite dimensional $E$, $\r{Mltp}(E)$ can be considered as an affine algebraic variety:
For any natural number $n$, let $\bb{C}^n$ denote the direct sum of $n$-copy of complex field $\bb{C}$.
Let $\alpha_1,\ldots,\alpha_n$ denote the standard basis of $\bb{C}^n$. Then any bilinear operator $\star\in\r{B}_2(\bb{C}^n)$
is exactly determined by a $n^3$-tuple $(\lambda_{ijk})_{1\leq i,j,k\leq n}$ of complex numbers such that
$\alpha_i\star\alpha_j=\sum_{k=1}^n\lambda_{ijk}\alpha_k$. It is easily checked that $\star$ belongs to $\r{Mltp}(\bb{C}^n)$ iff
the associated $n^3$-tuple $(\lambda_{ijk})$ is in the zero locus of the following class of quadratic polynomials:
$$\sum_\ell\lambda_{ij\ell}\lambda_{\ell kp}-\sum_q\lambda_{iqp}\lambda_{jkq}\hspace{10mm}(1\leq i,j,k,p\leq n).$$
In other word, $\r{Mltp}(\bb{C}^n)$ is identified with a Zariski-closed algebraic subset of affine space $\bb{C}^{n^3}$.
Similar to our result (Theorem \ref{1803191105}),
it has been shown in \cite{Gabriel1} that $\r{Mltp}_\r{u}(\bb{C}^n)$ is Zariski-open in $\r{Mltp}(\bb{C}^n)$.
(Note that the the $|\cdot|$-topology of $\r{B}_2(\bb{C}^n)$ coincides with the Euclidean topology of $\bb{C}^{n^3}$, and the Zariski
topology is strictly coarser than Euclidean topology.)

Let us observe that $\r{Mltp}(E)$ is a `large' subset of $\r{B}_2(E)$: First of all, note that
if $F$ be a $n$-dimensional vector space, then any semigroup $S$ with $n$ elements together with any vector basis $(\alpha_s)_{s\in S}$
of $F$, indexed by elements of $S$,  define a multiplication $\star\in\r{Mltp}(F)$ by $\alpha_s\star\alpha_{s'}=\alpha_{ss'}$.
Now, suppose that $E_1,E_2$ are Banach spaces and let $\star_1\in\r{Mltp}(E_1)$ and $\star_2\in\r{Mltp}(E_2)$.
Then one can define a multiplication $\star\in\r{Mltp}(E_1\oplus E_2)$ by
$$(x_1,x_2)\star(x'_1,x'_2):=(x_1\star_1 x_1',x_2\star_2 x'_2).$$
Thus, $(\star_1,\star_2)\mapsto\star$ defines an embedding $\r{Mltp}(E_1)\times\r{Mltp}(E_2)\to\r{Mltp}(E_1\oplus E_2)$.
These observations show that for an arbitrary Banach space $E$, one can define an element of $\r{Mltp}(E)$, simply by choosing a
$n$-dimensional subspace $E_1\subset E$, any topological complement $E_2$ of $E_1$ in $E$, any semigroup with $n$ elements,
and any vector basis of $E_1$.

In the following, we shall introduce a quotient $\bb{M}(E)$ of $\r{Mltp}(E)$ which is much smaller than $\r{Mltp}(E)$ and is more reasonable for study.

Let $\r{Hom}(E)\subset\r{B}(E)$
be the topological group of linear homeomorphisms from $E$ onto $E$, with $|\cdot|$-topology.
$\r{Hom}(E)$ acts on $\r{Mltp}(E)$ from left in a canonical way: Let $t\in\r{Hom}(E)$ and $*\in\r{Mltp}(E)$.
Then $t\cdot*=*_t$, the action of $t$ on $*$, is defined to be the continuous multiplication on $E$ given by $a*_tb:=t^{-1}(t(a)*t(b))$.
Note that the linear map $t^{-1}$ becomes a homeomorphic algebra isomorphism from $(E,*)$ onto $(E,*_t)$.

We show that the described group action is continuous: For any $s,t\in\r{B}(E)$ let
$\Lambda_t,\Gamma_{s,t}:\r{B}_2(E)\to\r{B}_2(E)$ be defined respectively by $\varphi\mapsto t\varphi,\varphi\mapsto\varphi(s,t)$.
Then these are bounded linear operators on $\r{B}_2(E)$, and it is not hard to see that $|\Lambda_t|=|t|$
and $|\Gamma_{s,t}|=|s||t|$. For $s,t\in\r{Hom}(E)$ and $\diamond,*\in\r{Mltp}(E)$ we have,
\begin{equation}\label{1803252023}
\begin{split}
|\diamond_s-*_t|&=|\Lambda_{s^{-1}}\Gamma_{s,s}(\diamond)-\Lambda_{t^{-1}}\Gamma_{t,t}(*)|\\
&\leq|\Lambda_{s^{-1}}\Gamma_{s,s}(\diamond)-\Lambda_{t^{-1}}\Gamma_{s,s}(\diamond)|+
|\Lambda_{t^{-1}}\Gamma_{s,s}(\diamond)-\Lambda_{t^{-1}}\Gamma_{t,t}(*)|\\
&\leq|\Lambda_{s^{-1}-t^{-1}}||\Gamma_{s,s}(\diamond)|+
|\Lambda_{t^{-1}}||\Gamma_{s,s}(\diamond)-\Gamma_{t,t}(*)|\\
&\leq|\diamond||s|^2|s^{-1}-t^{-1}|+|t^{-1}||\Gamma_{s,s}(\diamond)-\Gamma_{t,t}(*)|.
\end{split}
\end{equation}
We also have,
\begin{equation}\label{1803252051}
\begin{split}
&|\Gamma_{s,s}(\diamond)-\Gamma_{t,t}(*)|\\
\leq\quad&|\Gamma_{s,s}(\diamond)-\Gamma_{s,t}(\diamond)|+|\Gamma_{s,t}(\diamond)-\Gamma_{t,t}(\diamond)|
+|\Gamma_{t,t}(\diamond)-\Gamma_{t,t}(*)|\\
\leq\quad&|\diamond|(|s|+|t|)|s-t|+|t|^2|\diamond-*|.
\end{split}
\end{equation}
It follows from (\ref{1803252023}) and (\ref{1803252051}) that the canonical action $\r{Hom}(E)\times\r{Mltp}(E)\to\r{Mltp}(E)$
is continuous. We denote the orbit space $\r{Mltp}(E)/\r{Hom}(E)$ by $\bb{M}(E)$. It is appropriate to call
$\bb{M}(E)$ moduli space of multiplications on $E$.

A natural problem related to geometry of Banach spaces arises:
\begin{problem}
Describe the `geometry' of $\bb{M}(E)$.
\end{problem}
It is clear that if two Banach spaces $E,F$ are homeomorphic linear isomorphic then $\r{Mltp}(E)$ and $\bb{M}(E)$
are respectively homeomorphic with $\r{Mltp}(F)$ and $\bb{M}(F)$.
Thus, $\r{Mltp}(E)$ and $\bb{M}(E)$ (and their invariants) can be considered as invariants
of $E$ in category of topological vector spaces.
However, the study of $\bb{M}(E)$ is a hard work even if $E$ is finite dimensional.

In the following we remark some aspects of the geometry of $\bb{M}(E)$.
(For a multiplication $*\in\r{Mltp}(E)$, its orbit $\{t\cdot *:t\in\r{Hom}(E)\}$, mutually as an element in $\bb{M}(E)$ and as a subset of
$\r{Mltp}(E)$, is denoted by $\bar{*}$.)
\begin{remark}\label{1804011340}
\emph{
\begin{enumerate}
\item[(i)]
The only orbit $\bar{*}$ which is closed in $\r{Mltp}(E)$ is the orbit of the zero multiplication $0$:
We have $\bar{0}=\{0\}$ and thus $\bar{0}$ is closed. Let $*\neq0$ be in $\r{Mltp}(E)$. It is clear that $0$ does not belong to $\bar{*}$.
For every $n\geq1$, let $t_n:=n^{-1}\r{id}\in\r{Hom}(E)$. Then $|t_n\cdot *|\to0$. This shows that $0$ is a limit point of
$\bar{*}$. Therefore $\bar{*}$ is not closed in $\r{Mltp}(E)$ if $*\neq0$.
Note that since $0$ is in the closure of any orbit, $\bar{0}$ belongs to any nonempty closed subset of $\bb{M}(E)$.
\item[(ii)] A multiplication $*\in\r{Mltp}(E)$
is called  rigid \cite[Definition 3.13]{Jarosz1} if there exist $\epsilon>0$ such that for every $\diamond\in\r{Mltp}(E)$
with $|\diamond-*|<\epsilon$ there exists $t\in\r{Hom}(E)$ with $\diamond=t\cdot *$. Thus, $*$ is rigid iff
$\bar{*}$ is an open subset of $\r{Mltp}(E)$, or equivalently, $\bar{*}$ is an isolated point of $\bb{M}(E)$.
Johnson has shown \cite[Theorem 2.1]{Johnson1} that for $*\in\r{Mltp}(E)$ if the second and the third
Hochschild cohomology groups (\cite{Johnson2}) of $(E,*)$ with values in $E$ vanish, then $*$ is rigid.
\item[(iii)] The algebraic geometry of $\bb{M}(\bb{C}^5)$, has been considered by G. Mazzola in \cite{Mazzola1}.
He has shown that $\bb{M}_\r{u}(\bb{C}^5)$, the moduli space of unital algebras on $\bb{C}^5$, has exactly $59$ members.
Algebraic geometry of $\bb{M}(\bb{C}^n)$, for $n=2,3,4$ has been considered respectively in \cite{Bodin1,FialowskiPenkava1,FialowskiPenkava2}.
\item[(iv)] Finding numerical functions (with some continuity properties) on $\r{Mltp}(E)$ and $\bb{M}(E)$
may be useful. By Theorem \ref{1901272153}, the function $\star\mapsto\|e_\star\|$ is continuous on $\r{Mltp}_\r{u}(E)$.
It is easily seen that for a $\star\in\r{Mltp}(E)$, the Hochschild cohomology groups (\cite{Johnson2}) of $(E,\star)$
with values in $E$ and its topological dual, is only depends on $\bar{\star}$. Thus, for instance, the dimensions of these groups
can be considered as a function on $\bb{M}(E)$ with discrete values.
\end{enumerate}
}
\end{remark}
We end this note by some examples of families of continuous multiplications which the above theorems and inequalities may be
effectively applied to them.
\begin{example}\label{1901251419}
\emph{Let $X$ be a compact metric space and $\b{M}_X$ be the Banach space of complex-valued continuous functions on $X\times X$,
with $\sup$-norm. Any Borel complex measure $\mu$ on $X$, induces a $\star_\mu\in\r{Mltp}(\b{M}_X)$ defined by
$$(f\star_\mu g)(x,y):=\int_Xf(x,z)g(z,y)\r{d}\mu(z)\hspace{10mm}(f,g\in\b{M}_X,x,y\in X).$$
In the case that $\mu$ is a probability measure, $(\b{M}_X,\star_\mu)$ is a Banach algebra. Some algebraic aspects of this
class of Banach algebras has been considered in \cite{Sadr1}. It is easily seen that the family $\{\star_\mu\}_\mu$ is continuously
depends on $\mu$ as $\mu$ varies in the Banach space $\c{M}(X)$ of Borel complex measures, with total variation norm.
Another multiplication on $\b{M}_X$, is the point-wise multiplication which makes $\b{M}_X$ to a commutative C*-algebra.
Note that in the case that $X:=\{1,\ldots,n\}$ and $\mu$ is the counting measure, $(\b{M}_X,\star_\mu)$ is isomorphic to the algebra
of $n\times n$ matrixes.
}
\end{example}
\begin{example}
\emph{Let $X$ be a locally compact space and $\c{M}(X)$ be as in Example \ref{1901251419}.
Any continuous semigroup multiplication $s:X\times X\to X$ defines a $\star_s\in\r{Mltp}(\c{M}(X))$ as follows.
$$(\mu\star_s\nu)(B):=\mu\times\nu(S^{-1}(B))\hspace{10mm}(\mu,\nu\in\c{M}(X),B\subseteq X).$$
}
\end{example}
\begin{example}
\emph{Let $A$ and $B$ be Banach algebras. Consider the Banach space $A\oplus B$ with an $\ell_p$-norm ($1\leq p\leq\infty$).
Any continuous algebra homomorphism $T:B\to A$ induces a $\star_T\in\r{Mltp}(A\oplus B)$ defined by
$$(a,b)\star_T(a',b'):=(aa'+aT(b')+T(b)a,bb')\hspace{10mm}((a,b),(a',b')\in A\oplus B).$$
It is easily verified that the family $\{\star_T\}_T$ is continuously depends on $T$ when $T$ varies in the space of continuous algebra
homomorphisms as a subspace of the Banach space of bounded linear operators from $B$ to $A$. Some properties of algebras of the type
$(A\oplus B,\star_T)$ has been considered by some authors. See \cite{PourabbsRazi1} and references therein.
}
\end{example}
\bibliographystyle{amsplain}

\end{document}